\documentclass[12pt]{amsart}
\usepackage[utf8]{inputenc}
\usepackage{amsfonts}

\usepackage{amsthm}
\usepackage{amssymb}
\usepackage{verbatim}
\usepackage{amsmath}
\usepackage{amscd}
\usepackage{latexsym,dsfont}

\usepackage{bbm}
\usepackage{mathrsfs}

\usepackage[colorlinks=true, pdfstartview=FitV, linkcolor=blue,
citecolor=blue, urlcolor=blue]{hyperref}

\newcommand{\disp}{\displaystyle}
\newtheorem{thm}[equation]{\textbf{Theorem}}

\newtheorem{lem}[equation]{\textbf{Lemma}}

\newtheorem{rem}[equation]{\textbf{Remark}}

\newtheorem{defn}[equation]{\textbf{Definition}}
\newtheorem{hyp}[equation]{\textbf{Hypothesis}}
\newtheorem{exe}[equation]{\textbf{Example}}
\theoremstyle{remark}

\DeclareMathOperator{\Div}{Div}

\newcommand{\rn}{{\mathbb{R}^n}}

\calclayout

\allowdisplaybreaks

\numberwithin{equation}{section}

\title[Bounded weak solutions]
{Bounded weak solutions to elliptic PDE with  data in Orlicz spaces}

\author{David Cruz-Uribe, OFS and Scott Rodney}

\address{David Cruz-Uribe, OFS \\
Dept. of Mathematics \\
University of Alabama \\
 Tuscaloosa, AL 35487, USA}

\email{dcruzuribe@ua.edu}

\address{Scott Rodney\\
Dept. of Mathematics, Physics and Geology \\ 
Cape Breton University \\
Sydney, NS B1Y3V3, CA} 

\email{scott\_rodney@cbu.ca}

\thanks{D.~Cruz-Uribe is supported by \
  research funds from the Dean of the College of Arts \& Sciences, the
  University of Alabama. S.~Rodney is supported by the NSERC Discovery
  Grant program. The authors would like to thank Andrea Cianchi for
  calling to our attention several important references.}

\date{November 29, 2020}


\begin{document}

\begin{abstract}
 A classical regularity result is that non-negative solutions to the Dirichlet
 problem $\Delta u =f$ in a bounded domain $\Omega$, where $f\in L^q(\Omega)$, $q>\frac{n}2$, satisfy
 $\|u\|_{L^\infty(\Omega)} \leq C\|f\|_{L^q(\Omega)}$.  We extend this result in
 three ways:  we  replace the Laplacian with a degenerate elliptic
 operator; we show that we can take the data $f$ in an Orlicz space
 $L^A(\Omega)$  that lies strictly between $L^{\frac{n}{2}}(\Omega)$ and
 $L^q(\Omega)$, $q>\frac{n}2$; and we show that that we can replace
 the $L^A$ norm in the right-hand side by a smaller expression
 involving the logarithm of the ``entropy bump''
 $\|f\|_{L^A(\Omega)}/\|f\|_{L^{\frac{n}{2}}(\Omega)}$, generalizing a
 result due to Xu.
\end{abstract}
\maketitle

%%%%%%%%%%%%%%%%%%%%%%%%
%INTRODUCTION%%%%%%%%%%%
%%%%%%%%%%%%%%%%%%%%%%%%

\section{Introduction}

In this paper, we consider boundedness properties of weak
(sub)solutions to the following Dirichlet problem:
%
%\begin{equation}\label{dp}
%\begin{cases}\Div\left(Q\nabla u\right) =  fv
%      & x\in \Omega,\\
%u  = 0&x\in \partial \Omega.
%\end{cases}
%\end{equation}

\begin{eqnarray}\label{dp}
\left\{\begin{array}{rcll}
-\Div\left(Q\nabla u\right)&=&fv&\textrm{for }x\in\Omega\\
u&=&0&\textrm{for }x\in\partial\Omega
\end{array}\right.
\end{eqnarray}

Throughout this paper, $\Omega$ is a bounded domain (i.e., an open and connected subset) of
$\rn$ with $n\geq 3$, $Q=Q(x)$ is a non-negative definite, symmetric,
measurable matrix with $Q\in L^{1}_{loc}(\Omega)$, $v$ is a
weight (i.e., a non-negative, measurable function) on $\Omega$ such that $v\in
L^1(\Omega)$, and the data
function $f$ is in $ L^{1}_{loc}(\Omega)$.

When $Q$ is a uniformly elliptic matrix and $v(x)=1$, it is a classical
result (See Maz$^\prime$ya~\cite{MR0131054,MR0259329},
Stampacchia~\cite{MR192177} and Trudinger~\cite[Theorem~4.1]{MR369884},~\cite[Theorem~8.16]{GT}) that there is a constant $C>0$ such that if
$f\in L^q(\Omega)$, then
$$\|u\|_{L^\infty(\Omega)} \leq C\|f\|_{L^q(\Omega)}$$
for any non-negative weak subsolution $u\in H^1(\Omega)$ of \eqref{dp}
provided $q>\frac{n}{2}$.  Moreover, a counter-example shows that this bound
is sharp even for the Laplacian and we cannot take $q=\frac{n}{2}$.

The standard proof of this result uses the classical Sobolev
inequality,
$$\|\psi\|_{L^{\frac{2n}{n-2}}(\Omega)} \leq C\|\nabla
\psi\|_{L^2(\Omega)},$$
valid for any $\psi\in H^1_0(\Omega)$, combined
with Moser iteration.  The restriction $q>\frac{n}{2}$ is
naturally connected to the classical Sobolev gain factor
$\sigma = \frac{n}{n-2}=\left(\frac{n}{2}\right)'$. 

The goal of this paper is to generalize this result.  First we show
that this estimate can be improved by replacing the space
$L^q(\Omega)$ by an Orlicz space $L^A(\Omega)$ that lies strictly
between $L^{\frac{n}{2}}(\Omega)$ and $L^q(\Omega)$ for
$q>\frac{n}{2}$.  For
brevity, we will defer many definitions to
Section~\ref{section:prelim} below.

\begin{thm} \label{thm:classical}
  Let $Q$ be a uniformly elliptic matrix, and let $f\in L^A(\Omega)$,
  where $A(t)=t^{\frac{n}{2}}\log(e+t)^q$, $q>\frac{n}{2}$.  Then there exists a
  constant $C=C(n,q, Q)$ such that, given any non-negative weak subsolution $u$
  of

\begin{align*}
\left\{\begin{array}{rcll}
-\Div\left(Q\nabla u\right)&=&f&\textrm{for }x\in\Omega,\\
u&=&0&\textrm{for }x\in\partial\Omega,
\end{array}\right.
\end{align*}

%  \begin{equation*}
%\begin{cases}\Div\left(Q\nabla u\right) =  f
 %     & x\in \Omega,\\
%u  = 0&x\in \partial \Omega,
%\end{cases}
%\end{equation*}
%
we have the estimate
  \[ \|u\|_{L^\infty(\Omega)} \leq C\|f\|_{L^A(\Omega)}. \]
\end{thm}

\begin{rem}
  After completing this paper we learned that a somewhat more general
  version of Theorem~\ref{thm:classical} was proved by
Cianchi~\cite[Theorem~5]{MR1721824} using very different methods.
\end{rem}

\medskip

We will prove Theorem~\ref{thm:classical} as a special case of a more
general result for solutions of~\eqref{dp} that holds for a much
larger class of matrices $Q$.   We can allow $Q$ to be both degenerate
and singular, but must impose some restrictions on the largest and
smallest eigenvalues.  We encode these restrictions in two assumptions
on the integrability of the largest eigenvalue and on the existence of
an $L^2$ Sobolev inequality with gain.   We state them together as a general
hypothesis.

%%%%%%%%%%%%%%%%%%%%%%%%%
%SOBOLEV'S INEQUALITY%%%%
%%%%%%%%%%%%%%%%%%%%%%%%%

\begin{hyp} \label{sobolev} Given the matrix $Q$ and the weight $v\in
  L^1(\Omega)$, assume that for some constant $k>0$,
  \[ |Q(x)|_{op} = \sup\{ |Q(x)\xi| : \xi \in \rn, |\xi|= 1\} \leq
    kv(x) \; \text{a.e.} \]
Moreover, assume that
  there exist constants $\sigma=\sigma(n,Q, v,\Omega)>1,~C_0\geq 1$ such
  that for every $\psi\in Lip_0(\Omega)$
\begin{equation}\label{sob}
  \left(\int_\Omega | \psi(x) |^{2\sigma}~v(x)dx\right)^{\frac{1}{2\sigma}}
  \leq
  C_0\left(\int_\Omega \left|\sqrt{Q(x)}\nabla \psi(x)\right|^2~dx\right)^{\frac12}.
\end{equation}
\end{hyp}

%%END SOBOLEV

The first assumption, that $|Q|_{op}\leq kv$, in
Hypothesis~\ref{sobolev} is necessary to prove many of the necessary
properties of weak derivatives in the corresponding degenerate Sobolev
space.  The
second assumption, that inequality~\eqref{sob} holds, in
Hypothesis~\ref{sobolev} reduces to the classical Sobolev inequality
if $Q$ is a uniformly elliptic matrix in $\Omega$ and
$\sigma=\frac{n}{n-2}=(\frac{n}{2})'$.  It allows us to perform the
necessary De Giorgi iteration.

These assumptions hold in several important special cases.  If $v$
satisfies the Muckenhoupt $A_2$ condition,
\[ [v]_{A_2} = \sup_B
  \frac{1}{|B|}\int_B v(x)\,dx \frac{1}{|B|}\int_B v(x)^{-1}\,dx
  < \infty, \]
where the supremum is taken over all balls $B$ in $\rn$,
and if $Q$ satisfies the degenerate ellipticity condition
\[ \lambda v(x)|\xi|^2 \leq \langle Q(x)\xi,\xi\rangle \leq
  \Lambda  v(x)|\xi|^2, \]
where $0<\lambda\leq \Lambda<\infty$,  then $|Q|_{op}\leq \Lambda v$ and~\eqref{sob} holds.  (See~\cite{MR643158}.)  More generally,
suppose that $u$ and $v$ are a pair of weights such that  $u(x)\leq
v(x)$ a.e., $v$ satisfies a doubling condition, $u\in A_2$, and there
exists $\sigma>1$ such that given any balls $B_1\subset B_2 \subset
\Omega$,
\[ \frac{r(B_1)}{r(B_2)}
  \bigg(\frac{v(B_1)}{v(B_2)}\bigg)^{\frac{1}{2\sigma}}
  \leq C \bigg(\frac{u(B_1)}{u(B_2)}\bigg)^{\frac{1}{2}}, \]
and $Q$ satisfies the degenerate ellipticity condition
\[  u(x)|\xi|^2 \leq \langle Q(x)\xi,\xi\rangle 
  \leq v(x)|\xi|^2, \]
then $|Q|_{op}\leq  v$  and we have the Sobolev inequality
\[ \left(\int_\Omega | \psi(x) |^{2\sigma}~v(x)dx\right)^{\frac{1}{2\sigma}}
  \leq
  C_0\left(\int_\Omega \left|\nabla
      \psi(x)\right|^2~u(x)dx\right)^{\frac12}, \]
so again ~\eqref{sob} holds.  (See~\cite{MR805809}.)

\begin{rem}
In~\cite{CRR2}, the authors and Rosta proved that when $v=1$, with
minor additional hypotheses
the global Sobolev inequality \eqref{sob} follows from a weaker, local Sobolev
inequality,
\[ \left(\frac{1}{|B|}\int_B
    |\psi(x)|^{2\sigma}\,dx\right)^{\frac{1}{2\sigma}}
  \leq
  C \bigg[ \frac{r(B)}{|B|}\int_B |\sqrt{Q}\nabla \psi(x)|^2\,dx +
  \frac{1}{|B|}\int_B |\psi(x)|^2\,dx \bigg]^{\frac{1}{2}}, \]
that holds for all (sufficiently small) balls $B\subset \Omega$. 
\end{rem}

We can now state our main result, which is a generalization of
Theorem~\ref{thm:classical} to degenerate elliptic operators.  Again,
for precise definitions see Section~\ref{section:prelim}.

%%%%%%%%%%%%%%%%%%
%%MAIN I%%%%%%%%%%
%%%%%%%%%%%%%%%%%%

\begin{thm}\label{main0}
Given a weight $v$ and the non-negative definite, symmetric matrix
$Q$, suppose that Hypothesis~\ref{sobolev} holds for some $\sigma>1$.
 Let $A(t) = t^{\sigma'}\log(e+t)^q$ where $q>\sigma'$. If  $f\in
 L^A(v;\Omega)$,  then any non-negative weak subsolution ${\bf
   u}=(u,\nabla u)\in QH^1_0(v;\Omega)$ of  \eqref{dp} satisfies
\begin{equation}\label{b1}
\|u\|_{L^\infty(v;\Omega)} \leq C\|f\|_{L^A(v;\Omega)},
\end{equation}
where $C$ is independent of both ${\bf u}$ and $f$. 
\end{thm}

We originally conjectured that the exponent $q$ in Theorem~\ref{main0}
is sharp in general, but  we were not able to prove this or find a
counter-example.  We then learned that
Cianchi~\cite[Theorem~5]{MR1721824}, as a consequence of a more general
result, showed that in the classical setting when $Q$ is uniformly
elliptic and $\sigma=(\frac{n}{2})'$, we must have that
$q>\frac{n}{2}-1$.  We round out his result by giving a simple
counter-example proving that it is sharp for the Laplacian.

\begin{exe} \label{example:not-sharp-but-close}
Let $n\geq 3$, and let $\Omega=B(0,1)$.  Then there exists a function
$f\in L^A(\Omega)$, where $A(t)=t^{\frac{n}{2}}\log(e+t)^q$, $q<\frac{n}{2}-1$, such
that the non-negative weak solution of the Poisson equation
\begin{align*}
\left\{\begin{array}{rcll}
-\Delta u&=&f&\textrm{for }x\in\Omega,\\
u&=&0&\textrm{for }x\in\partial\Omega,
\end{array}\right.
\end{align*}
%\[
 %\begin{cases} \Delta u  =  f
   %   & x\in \Omega,\\
%u  = 0&x\in \partial \Omega,
%\end{cases} \]
%
is unbounded.
\end{exe}

\begin{rem}
  We  conjecture that the sharp exponent in Theorem~\ref{main0} is
  $q>\sigma'-1$.  However, the bound $q>\sigma'$ appears to be
  intrinsic to our proof, so either our proof needs to be refined or
  another approach is needed.  We note that the proof
  in~\cite{MR1721824} relies on re-arrangement estimates and so does
  not readily extend to the case of degenerate operators.
\end{rem}
\medskip

Our second main result shows that inequality~\eqref{b1} can be
sharpened so  that the right-hand side  only depends on the logarithm of the $L^A$
norm.  

%%%%%%%%%%%%%%%%%%
%%%%MAIN 2%%%%%%%%
%%%%%%%%%%%%%%%%%%

\begin{thm}\label{main1}
Given a weight $v$ and a non-negative definite, symmetric matrix
$Q$, suppose that Hypothesis~\ref{sobolev} holds for some $\sigma>1$.
 Let $A(t) = t^{\sigma'}\log(e+t)^q$, where $q>\sigma'$. If  $f\in
 L^A(v;\Omega)$,  then any non-negative weak subsolution ${\bf
   u}=(u,\nabla u)\in QH^1_0(v;\Omega)$ of  \eqref{dp} satisfies
\begin{equation} \label{eqn:fixed}
\|u\|_{L^\infty(v;\Omega)} \leq
C\|f\|_{L^{\sigma'}(v;\Omega)}
\left(1+\log\left(1+\frac{\|f\|_{L^A(v;\Omega)}}
    {\|f\|_{L^{\sigma'}(v;\Omega)}}\right)\right),
\end{equation}
where $C$ is independent of both ${\bf u}$ and $f$.
\end{thm}

  Theorem~\ref{main1} generalizes the main result of Xu~\cite{X}, but
  we note that there is a mistake in the statement of his main
  result.   Working in the same setting as
  Theorem~\ref{thm:classical}, he claims to show
  that
 \begin{equation} \label{eqn:xu}
\|u\|_{L^\infty(\Omega)} \leq
C\|f\|_{L^{\frac{n}{2}}(\Omega)}
\left(1+\log\left(1+{\|f\|_{L^q(\Omega)}}
   \right)\right),
\end{equation}
 where $q>\frac{n}{2}$.   However, a close examination of his proof
 shows that he only proves this when
 $\|f\|_{L^{\frac{n}{2}}(\Omega)}\geq 1$, and in fact what he proves
 is the analog of Theorem~\ref{main1}.   It is straightforward to see
 that \eqref{eqn:xu} cannot hold if
 $\|f\|_{L^{\frac{n}{2}}(\Omega)}<1$; if it did, then if we fix $f$
 and the corresponding solution $u$, then we could apply this
 inequality to $f/N$ ($N>1$ large) and $u/N$.  Then we could take the
 limit as $N\rightarrow \infty$ to conclude that $\|u\|_{L^\infty(\Omega)} \leq
C\|f\|_{L^{\frac{n}{2}}(\Omega)}$, which is false in general.  

\begin{rem}
  The ratio ${\|f\|_{L^A(v;\Omega)}}/
    {\|f\|_{L^{\sigma'}(v;\Omega)}}$ in Theorem~\ref{main1} measures
    how much bigger the Orlicz norm is than the associated Lebesgue
    space norm.   It is similar in spirit, though not in detail, to
    the ``entropy bump'' conditions introduced in the study of
    weighted norm inequalities in harmonic
    analysis~\cite{MR3357767, MR3539383}.
  \end{rem}
  \medskip

Our two main results are established via De Giorgi iteration on the
  level sets.  De Giorgi's original arguments are in
  \cite{DeGiorgi} but more helpful descriptions are found in
  \cite{MR1912731} and in~\cite{Korobenko:2016ue}, where
  De Giorgi iteration is applied in an infinitely degenerate elliptic
  regime.   We were unable to adapt Moser iteration to work in the
  context of Orlicz norms, and it remains an open question whether
  such an approach is possible in this setting.  

  \medskip

%%END THEOREMS

The remainder of the paper is organized as follows.  In
Section~\ref{section:prelim} we gather some preliminary results.  We
give a definition of Young functions and the associated Orlicz spaces, and
record some useful properties.  We then define weak
solutions to the Dirichlet problem. This definition has to
include  the possibility that the matrix $Q$ can be both degenerate
and singular, and    we give it in terms of a degenerate
Sobolev space, building upon results in~\cite{CRR1} and elsewhere.  We
prove a number of properties of weak derivatives in this setting;
we believe these results should be useful tools for other problems.  
We also prove that bounded, non-negative subsolutions of~\eqref{dp} must satisfy an
exponential integrability condition.  This result is a key lemma for
the proof of Theorem~\ref{main1} and is modeled on a similar result
due to Xu~\cite{X} in the classical setting.  For completeness we include the details of the proof.
In Section~\ref{section:main0} we prove Theorem~\ref{main0}; as noted
above, the proof
uses a version of De Giorgi iteration adapted to the scale of Orlicz
spaces.  This iteration argument was gotten by a careful adaptation of an
argument due Korbenko, {\em et al.}~\cite[Section~4.2]{Korobenko:2016ue}.
In
Section~\ref{section:main1} we prove Theorem~\ref{main1}; our proof is
a generalization of the argument in~\cite{X} and requires us to deal
with a number of technical obstacles.  Finally, in
Section~\ref{section:counter-example} we construct
Example~\ref{example:not-sharp-but-close}. 

%%%%%%%%%%%%%%%%%%%%
%%PRELIMINARIES%%%%%
%%%%%%%%%%%%%%%%%%%%

\section{Preliminaries}
\label{section:prelim}

In this section we gather some preliminary definitions and results.
We begin with some notation.  The constant $n$ will always denote the
dimension of the underlying space $\rn$.  By $C$, $c$, etc. we will
mean a constant that may change from appearance to appearance, but
whose value depends only on the underlying parameters.  If we want to
specify this dependence, we will write, for instance, $C(n,p)$, etc.
If we write $A\lesssim B$, we mean that there exists a constant $c$
such that $A\leq cB$.  If $A\lesssim B$ and $B\lesssim A$, we write
$A\approx B$.

A weight $v$ will always be a
non-negative, measurable function such that $v\in L^1(\Omega)$.  Given
a set $E\subset \Omega$, $v(E)=\int_E v(x)\,dx$.
Given a weight $v$, $L^p(v;\Omega)$ is the collection of all those measurable functions $g:\Omega \rightarrow \mathbb{R}$ for which
$$\|g\|_p = \|g\|_{L^p(v;\Omega)} = \displaystyle\left(\int_\Omega |g(x)|^p~v(x)dx\right)^{1/p} <\infty.$$

\subsection*{Orlicz spaces}
Our main hypothesis on the data function $f$ in \eqref{dp} is that it
belongs to the Orlicz space $L^A(v;\Omega)$.  Here we gather some
essential results about these spaces but we assume the reader has some
familiarity with them.  For complete information, we
refer to \cite{KR,RR}.  For a briefer summary, see
\cite[Chapter~5]{CMP}.  

By a Young function we mean a function $A : [0,\infty)\rightarrow
[0,\infty)$ that is continuous, convex, strictly increasing, 
$A(0)=0$, and $\frac{A(t)}{t}\rightarrow \infty$ as $t\rightarrow
\infty$.   Given a Young
function $A$, define $L^A(v;\Omega)$ to be the Banach space of
measurable functions $h:\Omega\rightarrow \mathbb{R}$ equipped with the
Luxembourg norm,
$$\|h\|_A = \|h\|_{L^A(v;\Omega)} = \inf\bigg\{\lambda>0~:~\int_\Omega
A\left(\frac{|f(x)|}{\lambda}\right)~v(x)dx \leq 1\bigg\}<\infty.$$

Given Young functions $A,\,B$ we can compare the associated norms by
appealing to a point-wise estimate.  We say that $A(t)\preceq B(t)$ if
there is a $t_0>0$ and a constant $c\geq 1$ depending only on $A,\,B$ so
that $A(t) \leq B(ct)$ for $t\geq t_0$.    For a proof of the
following result, see~\cite[Theorem~13.3]{KR} or \cite[Section~5.1]{RR}.

\begin{lem}\label{normcompare} Given Young functions $A,\,B$, if
  $A\preceq B$, then there exists a constant $C=C(A,B,v(\Omega))$ such
  that for every $f\in L^B(v;\Omega)$,
$$\|f\|_{L^A(v;\Omega)} \leq C\|f\|_{L^B(v;\Omega)}.$$
\end{lem}

%%%%%%%%
%%%%%%%%

Given a Young function $A$, we define the conjugate Orlicz function,
$\bar{A}$, via the pointwise formula
\[ \bar{A}(t) = \sup\{ st-B(s) : s> 0 \}. \]
The pair $A,\,\bar{A}$ satisfy a  version of H\"older's
inequality in the scale of Orlicz spaces.  If $f\in L^A(v;\Omega)$ and
$g\in L^{\bar{A}}(v;\Omega)$, then $fg\in L^1(v;\Omega)$ and
\begin{equation} \label{holders}
\int_\Omega |f(x)g(x)|v(x)\,dx \leq 2\|f\|_A\|g\|_{\bar{A}}.
\end{equation}

\medskip

In our main results we consider  Young functions of the form
\begin{equation} \label{eqn:log-bump}
  B(t) = t^p\log(e+t)^q,
\end{equation}
where $1<p,\,q<\infty$.  The inverse and conjugate functions
associated with these Young functions are well-known:  see, for
instance, \cite{CMP}.   We have that
\begin{gather}
\label{YoungB2}\bar{B}(t) \approx \disp\frac{t^{p'}}{\log(t)^q} \approx \disp\frac{t^{p'}}{\log(e+t)^q}, \\ 
\label{YoungB-1}\bar{B}^{-1}(t) \approx t^{1/p'}\log(e+t)^{\frac{q}{p}},
\end{gather}
where the implicit constants depend on $p,\,q$. 
As a consequence of Lemma~\ref{normcompare} we have the following
estimate which we will need below; details are straightforward and are
omitted.

\begin{lem}\label{SCALE}
  Let $1\leq p_1\leq p_2<\infty$, $1\leq q_1\leq q_2<\infty$ and define
  \[ A(t) = t^{p_1}\log(e+t)^{q_1}, \qquad B(t) =
    t^{p_2}\log(e+t)^{q_2}. \]
  Then, given $f\in L^{B}(v;\Omega)$,
  $$\|f\|_{L^{p_1}(v,\Omega)}
  \lesssim \|f\|_{L^{A}(v;\Omega)}
  \lesssim \|f\|_{L^{p_2}(v;\Omega)}
  \lesssim \|f\|_{L^{B}(v;\Omega)}.$$
  The implicit constants depend on $p_i$ and $q_i$, $i=1,\,2$, and $v(\Omega)$.
\end{lem}

We conclude this section with an estimate for the
$L^{\bar{B}}(\Omega)$ norm of an indicator function
$\mathbbm{1}_S$ for $S\subset\Omega$; this quantity plays an
essential role in our proofs of Theorems \ref{main0} and \ref{main1}.
This computation is well-known, but to make clear the dependence on
the constants we include its short proof.

%%%%%%%%%%%%%%%%%%%%%%%%
%ORLICZ NORM OF THE INDICATOR
%%%%%%%%%%%%%%%%%%%%%%%%

\begin{lem}\label{indicators}
Given the Young function $B$ defined by~\eqref{eqn:log-bump} then for
any  $S\subset\Omega$, $v(S)>0$,
\begin{equation}\label{indicator}
  \|\mathbbm{1}_S\|_{L^{\bar{B}}(v;\Omega)}
  \leq \disp\frac{cv(S)^{\frac{1}{p'}}}
  {\log(1+v(S)^{-1})^{\frac{q}{p}}},
\end{equation}
where $c=c(p,q)>0$.
\end{lem}

\begin{proof}  Given $B$, $\bar{B}$ is defined
  by~\eqref{YoungB2}.   Set
  $$F=\left\{\lambda>0~:~ \int_\Omega
    \bar{B}\left(\frac{\mathbbm{1}_S(x)}{\lambda}\right)v(x)\,dx \leq
    1\right\}$$

  and notice that $F\neq \emptyset$.  For each $\lambda \in F$,
\[ 
  v(S)~\bar{B}\left(\frac{1}{\lambda}\right)
  = \int_\Omega \bar{B}
  \left(\frac{\mathbbm{1}_S(x)}{\lambda}\right)v(x)\,dx
  \leq 1.
\]
Since $\bar{B}$ is invertible and increasing,
$$\lambda \geq \left[\bar{B}^{-1}\left(\frac{1}{v(S)}\right)\right]^{-1}=m_0>0.$$
Again by  the invertibility of $\bar{B}$, 
\[
  \int_\Omega \bar{B}\left(\frac{\mathbbm{1}_S(x)}{m_0}\right)v(x)\,dx
  = v(S)\bar{B}(m_0^{-1}) = 1. 
\]
Hence, $m_0\in F$, and it follows that
$\|\mathbbm{1}_S\|_{L^{\bar{B}}(\Omega)} = m_0$.  By
inequality~\eqref{YoungB-1},
$$m_0=\bar{B}^{-1}\left(\frac{1}{v(S)}\right)
\geq c(p,q)v(S)^{-\frac{1}{p'}}\log(e+v(S)^{-1})^{\frac{q}{p}}$$
and \eqref{indicator} follows.
\end{proof}
% 

%%%%%%%%%%%%%%%%%%%%%%%%
%WEAK SOLUTIONS%%%%%%%%%
%%%%%%%%%%%%%%%%%%%%%%%%

\subsection*{Degenerate Sobolev spaces and weak solutions}
We now give a precise definition of weak
(sub)solutions to the Dirichlet problem~\eqref{dp}.  This question has
been explored in a number of papers by ourselves and others:
see~\cite{CW,CRR1,CRR2,GT,MR,MRW,R,SW2}.  Here we sketch the relevant details.

Given a non-negative definite, symmetric and measurable matrix
function $Q$ on $\Omega$ and a weight $v \in L^1_{loc}(\Omega)$, the
solution space for the Dirichlet problem is the matrix weighted
Sobolev space $QH^1_0(v;\Omega)$.  This space is defined as the
abstract completion (in terms of Cauchy sequences) of the space
$Lip_0(\Omega)$ (i.e., Lipschitz functions with compact support in $\Omega$) with respect to the norm
$$\|\psi\|_{QH^1_0(v;\Omega)} = \|\psi\|_{L^2(v;\Omega)} + \|\nabla \psi\|_{L^2_Q(\Omega)},$$
where $L^2_Q(\Omega)$ is the Banach space of $\mathbb{R}^n$
vector-valued functions ${\bf g}$ on $\Omega$ that satisfy
$$\|{\bf g}\|_{L^2_Q(\Omega)}
= \bigg(\int_\Omega |\sqrt{Q(x)}{\bf g}(x)|^2~dx\bigg)^{\frac{1}{2}}<\infty.$$
This norm is well defined for $\psi\in Lip_0(\Omega)$  provided
$|Q|_{op}\in L_{loc}^1(\Omega)$; in particular, $QH^1_0(v;\Omega)$ is
well defined if the first assumption in Hypothesis~\ref{sobolev}
holds.

With this definition, the Sobolev space
$QH^1_0(v;\Omega)$ is a collection of equivalence classes of Cauchy
sequences of $Lip_0(\Omega)$ functions.  However, the spaces $L^2(v;\Omega)$ and
$ L^2_Q(\Omega)$ are complete:  for a proof that  $L^2_Q(\Omega)$ is
complete,  see~\cite{SW2}
or~\cite{CRR1} where it was proved that $L^p_Q(\Omega)$ is complete
for $1\leq p<\infty$.   Therefore, to each equivalence class
$[\{\psi_j\}]$ in  $QH^1_0(v;\Omega)$ we can associate a unique pair
${\bf u}=(u, {\bf g}) \in L^2(v;\Omega)\times L^2_Q(\Omega)$ whose
norm is given by
\begin{align*}
  \|{\bf u}\|_{QH^1_0(v;\Omega)}
  &= \|u\|_{L^2(v;\Omega)} + \|{\bf g}\|_{L^2_Q(\Omega)}\\
  &=\displaystyle\lim_{j\rightarrow\infty}\left(\|\psi_j\|_{L^2(v;\Omega)}
    + \|\nabla \psi_j\|_{L^2_Q(\Omega)}\right).
\end{align*}
Conversely, given a pair $(u,{\bf g})$ we will say that it is in
$QH^1_0(v;\Omega)$ if there exists a sequence $\{u_j\}_j\subset
Lip_0(\Omega)$ such that $(u_j \nabla u_j)$ converges to $(u,{\bf g})$
in $L^2(v;\Omega)\times L^2_Q(\Omega)$.

Hereafter, we will denote ${\bf g}$ by $\nabla u$ since $\bf g$ plays
the role of a weak gradient of $u$.  However, while we
adopt this formal notation we want to stress that the function ${\bf g}$ is
not the weak gradient of $u$ in the sense of classical Sobolev
spaces.  For further details, ~\cite{CRR1} contains the construction of $QH^{1,p}(v;\Omega)$ for $p\geq 1$.  Additionally, unweighted constructions of $QH^{1,p}_0(1;\Omega)$ are found in~\cite{CRR2,MRW} for $p\geq 1$ and~\cite{MR,R} for $p=2$.

We can extend the second assumption in Hypothesis \ref{sobolev} to functions in
$QH^1_0(\Omega)$; this follows from density of $Lip_0(\Omega)$
functions and we omit the proof.

%%%%%%%
%%SOBOLEV FOR QH^1_0
%%%%%%%

\begin{lem}\label{sobolev2}
Suppose the Sobolev inequality of Hypothesis~\ref{sobolev} holds.  Then,
\begin{equation}\label{sob2}
  \left(\int_\Omega | w(x) |^{2\sigma}v(x)\,dx
  \right)^{\frac{1}{2\sigma}}
  \leq C_0\left(\int_\Omega \left|\sqrt{Q(x)}\nabla w\right|^2\,dx
  \right)^{\frac{1}{2}} 
\end{equation}
for every ${\bf w}=(w,\nabla w)\in QH^1_0(v;\Omega)$ where $C_0$
is the same as in Hypothesis~\ref{sobolev}.
\end{lem}

\medskip

We can now define the weak solution to the Dirichlet problem. 
%%%%%
%DEFN OF NON-NEG WEAK SOLUTION
%%%%%

\begin{defn}\label{weaksol}
  A pair ${\bf u} = (u,\nabla u)\in QH^1_0(\Omega)$ is said to be
  a weak solution of the Dirichlet problem \eqref{dp}
  if
  $$\int_\Omega \nabla \psi(x) \cdot Q(x)\nabla u (x) \,dx
  = \int_\Omega f(x)\psi(x)v(x)\,dx$$
for every  $\psi\in Lip_0(\Omega)$.  The pair is said to be a
non-negative weak subsolution  if $u(x)\geq 0$ $v$-a.e. and
$$\int_\Omega \nabla \psi(x) \cdot Q(x)\nabla u (x) \,dx
\leq  \int_\Omega f(x)\psi(x)v(x)\,dx$$
for every non-negative  $\psi\in Lip_0(\Omega)$. 
\end{defn}

Note that if ${\bf h} = (h,\nabla h)\in QH^1_0(v;\Omega)$ with $h(x)\geq 0$ $v$-a.e., then by a
standard limiting argument we may use $h$ as our test function in
Definition~\ref{weaksol}.

\begin{rem}
  The existence of weak solutions to \eqref{dp} when  $v=1$ was studied in
  \cite{GT,MR,R}, and when   $v(x)=|Q(x)|_{op}$ in~\cite{CRR1,CRR2}. 
\end{rem}
%%%%%%%
%%%%%%%
%%%%%%%
%%%%%%%

\subsection*{Properties of weak gradients}
We now develop some useful properties of functions in the degenerate
Sobolev space $QH^1_0(v;\Omega)$.  All of these properties are well
known in the classical case:  see, for instance,~\cite{GT}.   In the
degenerate case, we stress that the first assumption in
Hypothesis~\ref{sobolev} is critical in proving these results and
throughout this subsection we assume that $v\in L^1(\Omega)$ and
$|Q|_{op}\leq kv$ a.e.

Our first result shows that weak gradients are zero almost everywhere
on sets of $v$-measure zero.

%%%%%%%%%%%%
%%%meas 0%%%%%
%%%%%%%%%%%

\begin{lem}\label{meas0} Let ${\bf u}=(u,\nabla u)\in
  QH^1_0(v;\Omega)$ and $w\in Lip_0(\Omega)$.  Then, given any set $E$
  of $v$-measure zero, we have that:
\begin{enumerate}
\item $\|\nabla w\|_{L^2_Q(E)}=0$;
\item $\|\nabla u\|_{L^2_Q(E)} = 0$;
\item $\sqrt{Q(x)}\nabla u(x)=0=\sqrt{Q(x)}\nabla w(x)$ a.e. $x\in E$. 
\end{enumerate}
\end{lem}

\begin{proof} If $w\in Lip_0(\Omega)$,   then $\nabla w$ is
defined a.e. in $\Omega$ by the Rademacher-Stepanov theorem and is in
$L^\infty$.  Therefore,  for a.e. $x\in \Omega$,
\[|\sqrt{Q(x)}\nabla w(x)| \leq |Q(x)|_{op}^{\frac12} |\nabla w(x)|
  \leq c v(x)^{\frac12}  |\nabla w(x)|;\]
hence,
\[\|\nabla w(x)\|_{L^2_Q(E)}^2 \leq \|\nabla w\|_\infty^2 v(E) = 0,\]
which proves (1).

Let ${\bf u} \in QH^1_0(v;\Omega)$; then there exists  a sequence
$\{w_j\}_j\subset Lip_0(\Omega)$ such that  $\nabla w_j \rightarrow
\nabla u$ in $L^2_Q(\Omega)$.  Then by the previous argument,
\[\|\nabla u\|_{L^2_Q(E)} =
  \lim_{j\rightarrow\infty}\|\nabla w_j\|_{L^2_Q(E)} = 0,\]
and so (2) holds.

Finally, (3) follows immediately from (1) and (2).
\end{proof}

Our second result shows that non-negative truncations of functions in
$QH^1_0(v;\Omega)$ are again in this space.

\begin{lem} \label{lemma:main1-lemma}
  Let ${\bf u}=(u,\nabla u)\in
  QH^1_0(v;\Omega)$ and fix $r>0$.  If $S(r)=\{ x\in \Omega : u(x)>r \}$,
  then $((u-r)_+, \mathbbm{1}_{S(r)}\nabla u) \in QH^1_0(v;\Omega)$.  
\end{lem}

\begin{proof}
  By the definition of $QH^1_0(v;\Omega)$ there exists a sequence
  $\{u_j\}_j$ in $Lip_0(\Omega)$ such that $u_j\rightarrow u$ in
  $L^2(v;\Omega)$ and $\nabla u_j \rightarrow \nabla u$ in
  $L^2_Q(\Omega)$. If we pass to a subsequence, we may assume that
  $u_j\rightarrow u$ pointwise $v$-a.e.  We will first prove that
  $(u_j-r)_+\rightarrow (u-r)_+$ in $L^2(v;\Omega)$.

  Define $f_j= |(u_j-r)_+ - (u-r)_+|^2$; then $f_j\rightarrow 0$
  $v$-a.e.  We will show that $f_j$ converges to $0$ in
  $L^1(v;\Omega)$; this follows from the generalized dominated convergence
  theorem~\cite[p.~59]{MR1681462} if we  show that there
  exist non-negative functions $g_j,\,g \in L^1(v; \Omega)$ such that
  $f_j\leq g_j$ and $\|g_j\|_{L^1(v;\Omega)}\rightarrow
  \|g\|_{L^1(v;\Omega)}$ as $j\rightarrow \infty$.  But we have that
\[ 
    f_j \leq 2(u_j-r)^2_+ +2(u-r)_+^2 \leq 4(|u_j|^2+r^2)+4(|u|^2+r^2)
    = g_j.  \]
Moreover, $g_j$ converges pointwise a.e. to $g=8(|u|^2+r^2)$, and
$g,\,g_j\in L^1(v;\Omega)$ since $v(\Omega)<\infty$.  Finally, since
$u_j\rightarrow u$ in $L^2(v;\Omega)$ we get that $\|g_j\|_{L^1(v;\Omega)}\rightarrow
  \|g\|_{L^1(v;\Omega)}$.

\smallskip

Now define $S_j=\{ x\in \Omega : u_j (x)>r \}$; then
$\mathbbm{1}_{S_j}\rightarrow \mathbbm{1}_{S}$ $v$-a.e. Moreover, we
have that $\nabla(u_j-r)_+ = \nabla u_j \mathbbm{1}_{S_j}$
a.e.~\cite[Lemma~7.6]{GT} and so $v$-a.e.  By passing to another
subsequence, we assume that
$\nabla u_j \mathbbm{1}_{S_j}\rightarrow \nabla u \mathbbm{1}_{S}$
pointwise $v$-a.e.  We claim that they converge in $L^2_Q(\Omega)$ as
well.  If this is the case, then we have shown that
$((u_j-r)_+, \nabla(u_j-r)_+)$ is Cauchy in $QH^1_0(v;\Omega)$, and
the desired conclusion follows at once.

To prove $L^2_Q(\Omega)$ convergence, note that
\[ \|\nabla u_j \mathbbm{1}_{S_j}- \nabla u
  \mathbbm{1}_{S}\|_{L^2_Q(\Omega)}
    \leq 
    \|\nabla u_j \mathbbm{1}_{S_j}- \nabla u
    \mathbbm{1}_{S_j}\|_{L^2_Q(\Omega)}
      +
    \|\sqrt{Q}\nabla
    u(\mathbbm{1}_{S_j}-\mathbbm{1}_{S})\|_{L^2(\Omega)}. \]
The first term on the right-hand side is less than $\|\nabla u_j -
\nabla u\|_{L^2_Q(\Omega)}$ which goes to $0$ as $j\rightarrow
\infty$.  To estimate the second term, let $E$ be the set of $x\in
\Omega$ where
$\mathbbm{1}_{S_j}(x)$  does not converge to $\mathbbm{1}_{S}(x)$.  Then
$v(E)=0$, and so by Lemma~\ref{meas0},
\[ \int_E |\sqrt{Q}\nabla u (\mathbbm{1}_{S_j}-\mathbbm{1}_{S})|^2\,dx
  \leq  \int_E |\sqrt{Q}\nabla u|^2 \,dx = 0.  \]
Since $|\sqrt{Q}\nabla u (\mathbbm{1}_{S_j}-\mathbbm{1}_{S})|\leq
|\sqrt{Q}\nabla u|\in L^2(\Omega)$, by the dominated convergence
theorem we have that as $j\rightarrow 0$,
\[ \|\sqrt{Q}\nabla
  u(\mathbbm{1}_{S_j}-\mathbbm{1}_{S})\|_{L^2(\Omega)}
  = \|\sqrt{Q}\nabla
    u(\mathbbm{1}_{S_j}-\mathbbm{1}_{S})\|_{L^2(\Omega\setminus E)}
    \rightarrow 0. \]
\end{proof}

%%%%%%%%%%%%%%%%% 

%%special sequences %%%%%%
%%%%%%%%%%%%%%%%%
%%%%%%%%%%%%%%%%%

Our next lemma proves the existence of an approximating sequence of
Lipschitz functions with some additional useful properties. 

\begin{lem} \label{specialapprox}
Let ${\bf u} = (u,\nabla u)\in QH^1_0(v;\Omega)$ with $u\in
L^\infty(v;\Omega)$ and $u\geq 0$ $v$-a.e. Then there exists a
sequence $\{u_j\}_j\in Lip_0(\Omega)$ such that:
\begin{enumerate}
\item $0\leq u_j(x)\leq \|u\|_{L^\infty(v;\Omega)}+1$ in $\Omega$;
\item $u_j\rightarrow u$ $v$-a.e. and also in $L^2(v;\Omega)$;
\item $\nabla u_j \rightarrow \nabla u$ in $L^2_Q(\Omega)$ and
  $\left|\sqrt{Q}\nabla u_j\right| \rightarrow \left|\sqrt{Q}\nabla
    u\right|$ pointwise a.e.;
\item $\|\nabla u_j\|_{L^2_Q(\Omega)}\leq \|\nabla u\|_{L^2_Q(\Omega)}+1$ for each $j\in \mathbb{N}$.
\end{enumerate}
\end{lem} 

\begin{proof}
  By the definition of $QH^1_0(v;\Omega)$ and by passing twice to a
  subsequence, there exists a sequence $\{z_j\}_j\subset
  Lip_0(\Omega)$ such that:
\begin{enumerate}
\item[($1^\prime$)] $z_j\rightarrow u$ both $v$-a.e. and also in $L^2(v;\Omega)$
\item[($2^\prime$)] $\nabla z_j\rightarrow \nabla u$ in $L^2_Q(\Omega)$ and $|\sqrt{Q}\nabla z_j| \rightarrow |\sqrt{Q}\nabla u|$ a.e.;
\item[($3^\prime$)] $\|\nabla z_j\|_{L^2_Q(\Omega)} \leq \|\nabla u\|_{L^2_Q(\Omega)}+1.$
\end{enumerate}

Now let $w_j = |z_j|$. Since  $u$ is non-negative $v$-a.e. in
$\Omega$, by  the triangle inequality we have that
\[|w_j-u| = ||z_j|-| u|| \leq |z_j-u|\]
$v$-a.e.  Therefore, we have that $w_j$ converges to $u$ both in $L^2(v;\Omega)$ and pointwise $v$-a.e. 

By the Rademacher-Stepanov theorem~\cite{MR3409135}, 
$\nabla w_j (x) = sgn(z_j(x))\nabla z_j(x)$ a.e.  Hence, $|\sqrt{Q}
\nabla w_j(x)| = |\sqrt{Q}\nabla z_j(x)|$ a.e. and so  $\|\nabla
w_j\|_{L^2_Q(\Omega)}\rightarrow \|\nabla u\|_{L^2_Q(\Omega)}$ as
$j\rightarrow \infty$. Thus $w_j\geq 0$ a.e. and properties ($1^\prime$)--($3^\prime$)
above hold with $z_j$ replaced by $w_j$.

\smallskip

We now define the sequence of $Lip_0(\Omega)$ functions $\{u_j\}_j$.
Set $M=\|u\|_{L^\infty(v;\Omega)}+1$ and let $\phi :
[0,\infty)\rightarrow [0,\infty)$ be such that $\phi\in C^\infty$,
$\phi$ is increasing, $\phi(x)=x$ if $0\leq x\leq M-\frac{1}{2}$,  $\phi(x)=M$
if $x\geq M+1$, and $\phi'(x)\leq 1$.   Define the $u_j$ by
$u_j(x)=\phi(w_j(x))$.  Then $\phi_j \in Lip_0(\Omega)$; moreover, $\nabla u_j(x) =
\phi'(w_j(x))\nabla w_j(x)$ a.e. and so 
\begin{equation} \label{eqn:grad-uj-bound}
 |\sqrt{Q(x)}\nabla u_j(x)| \leq |\sqrt{Q(x)}\nabla w_j(x)|. 
\end{equation}

We claim that $\{u_j\}_j$ satisfies properties (1)--(4) above.
By the definition of $\phi$, $0\leq u_j \leq M$, so property (1)
holds.   Property (4) follow immediately from
\eqref{eqn:grad-uj-bound} and property ($3^\prime$) for the $w_j$.

It remains to prove properties (2) and (3).   By the choice of $M$,
$u(s) \leq M-1$ for $v$-a.e. $s\in \Omega$.  We also have that
$w_j(s)\rightarrow u(s)$ $v$-a.e.  Let $F$ be the set of all  $s\in \Omega$ such that both of these
hold. Then $v(\Omega\setminus F)=0$.  Given $s\in F$ there exists $N>0$ such that if $j \geq N$, 
$w_j(s)<M-\frac{1}{2}$, and so $u_j(s)=w_j(s)$.  Thus, $u_j\rightarrow
u$ pointwise $v$-a.e.  Since $u$ is bounded and $v(\Omega)<\infty$, by
the dominated convergence theorem we also have that $u_j\rightarrow u$
in $L^2(v;\Omega)$.  This proves (2).

To prove (3) define the set $F$ as above.  For each $s\in F$, there exists
$N>0$ such that for each $j\geq N$ there exists a
ball $B_{j,s}$ where for $x\in B_{j,s}$, $w_j(x) < M-\frac{1}{2}$;
hence, $\nabla u_j(s) = \nabla w_j(s)$ for $j\geq N$. 
Now let $G$ be the set of $s\in
\Omega$ such that $|\sqrt{Q(s)}\nabla
w_j(s)|\rightarrow |\sqrt{Q(s)}\nabla u(s)|$; by ($2'$),
$|\Omega\setminus G|=0$.  Since $v\,dx$ is an absolutely continuous
measure, $v(\Omega\setminus G)=0$.   Let $H=F\cap G$.  Then on $H$ we
have that $|\sqrt{Q}\nabla
u_j|\rightarrow |\sqrt{Q}\nabla u|$ pointwise.  But $v(\Omega\setminus
H)=0$ so by Lemma~\ref{meas0} we have
that
\[ \|\nabla u_j \|_{L^2_Q(\Omega\setminus H)} = 0 = \|\nabla u
  \|_{L^2_Q(\Omega\setminus H)}. \]
This implies that $|\sqrt{Q}\nabla u_j| = 0 = |\sqrt{Q}\nabla u|$ almost
everywhere on $\Omega\setminus H$.  Therefore, we have that $|\sqrt{Q}\nabla
u_j|\rightarrow |\sqrt{Q}\nabla u|$ pointwise a.e.

Finally, to prove that $\nabla u_j \rightarrow \nabla u$ in
$L^2_Q(\Omega)$ we use the generalized dominated convergence theorem
as in the proof of Lemma~\ref{lemma:main1-lemma}.  Let $f_j=
|\sqrt{Q}(\nabla u_j-\nabla u)|^2$; then $f_j\rightarrow 0$ a.e.
Further, by \eqref{eqn:grad-uj-bound}
\[ f_j \leq 2|\sqrt{Q}\nabla u_j|^2 +2|\sqrt{Q}\nabla u|^2
  \leq  2|\sqrt{Q}\nabla w_j|^2 +2|\sqrt{Q}\nabla u|^2 = g_j.  \]
Again by ($2^\prime$), $g_j \rightarrow 4|\sqrt{Q}\nabla u|^2=g$ a.e.,
and since $\nabla w_j \rightarrow \nabla u$ in $L^2_Q(\Omega)$,
$g_j\rightarrow g$ in $L^1(\Omega)$.   Therefore, $f_j \rightarrow 0$
in $L^1(\Omega)$, which completes the proof of (3).
\end{proof}

The next two lemmas give the product rule and chain rule associated to
pairs in $QH^1_0(\Omega)$.   The proofs are adapted from those of
similar results in~\cite{MRW}.

%%%%%%%%%%%
%%product rule%%%
%%%%%%%%%%%

\begin{lem}\label{prod} Let $(u,\nabla u) \in QH_0^1(v;\Omega)$ and
  let $\psi\in Lip_0(\Omega)$.  Then we have that  $(u\psi,\psi\nabla u+u\nabla \psi) \in QH^1_0(v;\Omega)$.
\end{lem}

\begin{proof}
  By the definition of  $QH_0^1(v;\Omega)$ there exists a sequence
  $\{w_j\}\subset Lip_0(\Omega)$ such that $w_j\rightarrow u$ in
  $L^2(v;\Omega)$ and $\nabla w_j \rightarrow \nabla u$ in
  $L^2_Q(\Omega)$.  But then we immediately have that
  \[ \|w_j\psi-u\psi\|_{L^2(v;\Omega)} \leq \|\psi \|_\infty
    \|w_j-u\|_{L^2(v;\Omega)}, \]
    and so $w_j\psi\rightarrow u\psi$ in $L^2(v;\Omega)$.

    Similarly, since $|Q|_{op}\leq kv$ a.e., we have that
    \begin{multline*}
      \|\nabla(w_j\psi) - (u\nabla\psi+\psi\nabla u)\|_{L^2_Q(\Omega)}
      \leq \|\psi \nabla w_j-\psi\nabla u\|_{L^2_Q(\Omega)}
      + \|w_j\nabla\psi - u\nabla \psi\|_{L^2_Q(\Omega)} \\
      \leq \|\psi\|_\infty \|\nabla w_j - \nabla u \|_{L^2_Q(\Omega)}
      + k\|\nabla \psi\|_\infty \|w_j-u\|_{L^2(v;\Omega)}. 
      \end{multline*}
   Thus, $\nabla(w_j\psi)\rightarrow u\nabla\psi+\psi\nabla u$ in
      $L^2_Q(\Omega)$ and so
      $(u\psi,u\nabla \psi + \psi\nabla u) \in QH^1_0(v;\Omega)$.
\end{proof}
%%%%%%%%%%% 

%%chain rule%%%%
%%%%%%%%%%%

\begin{lem}\label{comp} Let $(u,\nabla u)\in QH^1_0(v;\Omega)$ with
  $u\geq 0$ $v$-a.e. and $u\in L^\infty(v;\Omega)$.  Then, given any
  non-negative function $\varphi\in C^1(\mathbb{R})$ such that
  $\varphi(0)=0$, the pair
  $(\varphi(u),\varphi'(u)\nabla u)\in QH^1_0(v;\Omega)$.
\end{lem}

\begin{proof}
  Let $\{u_j\}_j\subset Lip_0(\Omega)$ be the sequence associated with
  $(u,\nabla u)$ given by Lemma \ref{specialapprox}. Since $u_j$ is
  Lipschitz with compact support in $\Omega$ and $\varphi(0)=0$,
  $\psi_j=\varphi(u_j)\in Lip_0(\Omega)$.  Since $u_j\rightarrow u$
  $v$-a.e., the continuity of $\varphi$ implies that
  $\psi_j\rightarrow \varphi(u) =\psi$ $v$-a.e.    By the
  fundamental theorem of calculus,
  \[ |\varphi(t)|
    = \left|\int_0^t \varphi'(s)ds\right|
    \leq \|\varphi'\|_{L^\infty([0,M])}|t| = A_0|t|\]
whenever $0\leq t\leq M =\|u\|_{L^\infty(v;\Omega)}+1$.   

Since by assumption and property (1) of  Lemma \ref{specialapprox},
$0\leq u(x),\, u_j(x) \leq M$  for $v$-a.e. $x\in \Omega$, we have
that $v$-a.e.,
\[ |\psi_j - \psi|^2
  \leq  2(|\psi_j|^2+|\psi|^2) 
\leq  2A_0^2(|u_j|^2 + |u|^2). \]
Since $|u_j|^2 + |u|^2 \rightarrow 2|u|^2$ $v$-a.e. and in
$L^1(v;\Omega)$, by the generalized Lebesgue dominated convergence
theorem we get that $\psi_j\rightarrow \psi$ in $L^2(v;\Omega)$.

To show the convergence of the gradients, first note that
$\sqrt{Q}\nabla \psi_j = \varphi'(u_j)\sqrt{Q}\nabla u_j$ a.e. in
$\Omega$ and so by the continuity of $\varphi'$ and property (3) in
Lemma \ref{specialapprox} we get that
$\sqrt{Q}\nabla \psi_j \rightarrow \varphi'(u)\sqrt{Q}\nabla u$ a.e.
Moreover,
\begin{multline*}
  |\sqrt{Q}\nabla (\psi_j) - \varphi'(u)\sqrt{Q}\nabla u|^2
  \leq 2|\varphi'(u_j) \sqrt{Q} \nabla u_j|^2 + 2|\varphi'(u) \sqrt{Q} \nabla u|^2 \\
  \leq 2A_0^2(|\sqrt{Q}\nabla u_j|^2 + |\sqrt{Q}\nabla u|^2).
\end{multline*}
The right-hand term converges to $4A_0^2|\sqrt{Q}\nabla u|^2$ both
pointwise a.e. and in $L^1(\Omega)$.  Therefore, we can again apply
the  generalized dominated
convergence theorem to get that
$\nabla \psi_j\rightarrow \varphi'(u)\nabla u$ in $L^2_Q(\Omega)$.
We conclude that
$(\varphi(u), \varphi'(u)\nabla u)\in QH^1_0(v;\Omega)$.
\end{proof}

%%%%%%%%%%%%%%%%%
%%%aux DP%%%%%%%%%%%
%%%%%%%%%%%%%%%%%
%%%%%%%%%%%%%%%%%

\subsection*{Exponential results}
In this section we give two results which are needed to prove
Theorem~\ref{main1}.   The first gives a solution to an auxiliary Dirichlet problem
and is  an application of the previous two lemmas.

\begin{lem}\label{AuxProb}
 Fix $\alpha>0$.  If $(u,\nabla u)\in QH^1_0(\Omega)$ is a
  non-negative bounded weak subsolution of the Dirichlet problem
\begin{eqnarray}\label{dpA}
\left\{\begin{array}{rcll}
-\Div\left(Q\nabla u\right)&=&fv&\textrm{for }x\in\Omega,\\
u&=&0&\textrm{for }x\in\partial\Omega,
\end{array}
\right.
\end{eqnarray}
then $(w,\nabla w) = (e^{\alpha u}-1,\alpha e^{\alpha u}\nabla u) \in QH_0^1(v;\Omega)$ is a non-negative
weak subsolution of the Dirichlet problem
\begin{eqnarray}\label{dpB}
\left\{ \begin{array}{rcll}
-\Div\left(Q\nabla w\right)&=&\alpha f(w+1)v&\textrm{for }x\in\Omega,\\
w&=&0&\textrm{for }x\in\partial\Omega.
\end{array}
\right.
\end{eqnarray}
\end{lem}

\begin{proof} Fix a non-negative $\psi\in Lip_0(\Omega)$.  By our
  assumptions on $(u,\nabla u)$ and by Lemmas~\ref{prod}
  and~\ref{comp} we have that
  that both
  $(w,\nabla w)= (e^{\alpha u}-1,\alpha e^{\alpha u}\nabla
  u)$ and $(\psi (w+1), (w+1)\nabla\psi+\psi\nabla w)$ are in
  $QH^1_0(\Omega)$. Since  $\nabla w = \alpha (w+1)\nabla u$ and
  $(u,\nabla u)$ is a non-negative weak subsolution of
  \eqref{dpA}, we have that
\begin{align*}
  \int_\Omega f(w+1)\psi~vdx
  &\geq  \int_\Omega \nabla (\psi(w+1))Q\nabla u~dx\\
& = \int_\Omega (w+1)\nabla \psi Q\nabla u~dx + \int_\Omega \psi\nabla(w+1)Q \nabla w~dx \\
  &= \frac{1}{\alpha}\int_\Omega \nabla \psi Q\nabla w~dx
    + \int_\Omega \psi \nabla w Q\nabla w~dx\\
&\geq \frac{1}{\alpha}\int_\Omega \nabla \psi Q\nabla w~dx.
\end{align*}
Since $\psi\in Lip_0(\Omega)$ is arbitrary, we conclude that $w$ is a
non-negative weak subsolution of~\eqref{dpB}.
\end{proof}

%%%%%%%
%%%%%%%
%%%%%%%
%%%%%%%

Our second result gives the exponential
integrability of bounded solutions to \eqref{dp}.   A version of this
result is proved in~\cite[Lemma~B]{X} for
uniformly elliptic operators; a qualitative version appeared
previously in~\cite[Example~4]{MR1721824}.  Here we adapt the proof
from \cite{X} to our more
general setting.

%%%%%%%%%%%%%%%%%
%EXPONENTIAL INTEGRABILITY
%%%%%%%%%%%%%%%%%

\begin{lem}\label{expint}
  Suppose  Hypothesis~\ref{sobolev} holds.  Let  $f\in L^{\sigma'}(v;\Omega)$ satisfy
  $\|f\|_{\sigma';v}\leq 1$, and let 
  $(u,\nabla u) \in QH_0^1(\Omega)$ be a bounded, non-negative weak
  subsolution of ~\eqref{dp}.  Then, for every
  $\gamma \in (0,\frac{4}{C_0^2})$, with $C_0$ as in
  \eqref{sob}, there 
  $M= M(\gamma,C_0, v(\Omega))$ such that 
\begin{equation}\label{expest}
  \int_\Omega e^{\gamma u(x)}v(x)\,dx
  \leq M.
\end{equation}
\end{lem}

\begin{proof}
  Let $f$ and 
  $(u,\nabla u)$ be as in the hypotheses.
  Define  $\varphi = e^{\gamma u} -1$ and $\psi = e^\frac{\gamma u}{2}-1$ with
  $\gamma>0$ to be chosen below.  Since $u$ is bounded, by
  Lemma~\ref{comp} we have that
  \[ (\varphi,\nabla \varphi) = ( e^{\gamma u} -1, \gamma e^{\gamma
      u}\nabla u), \quad
    (\psi,\nabla \psi) = (e^\frac{\gamma u}{2}-1,
    \frac{\gamma}{2}e^{\frac{\gamma u}{2}} \nabla u) \]
  are in $QH_0^1(\Omega)$.  Further, we immediately have the
  identities $\varphi =\psi^2 + 2\psi$,
  $\nabla \psi = \frac{\gamma}{2}e^{\frac{\gamma u}{2}}\nabla u$, and
  $\nabla \varphi =2 e^{\frac{\gamma u}{2}}\nabla \psi$.
  If we apply the Sobolev inequality~\eqref{sob2} and use $\varphi$ as
  a test function in Definition~\ref{weaksol} we can estimate as follows:
\begin{align*}
  \|\psi\|_{L^{2\sigma}(v;\Omega)}^2
  &\leq
C_0^2\int_\Omega \left|\sqrt{Q(x)}\nabla \psi(x)\right|^2\,dx\\
  & =
    \frac{C_0^2\gamma}{4}\int_\Omega
    \nabla\varphi(x) \cdot Q(x)\nabla u(x)\,dx\\
  &\leq \frac{C_0^2\gamma}{4}\int_\Omega
    f(x)\varphi(x) v(x)\,dx\\
  &= \frac{C_0^2\gamma}{4}
    \left( \int_\Omega f(x)\psi(x)^2 v(x)\, dx
    + 2\int_\Omega f(x)
    \psi(x) v(x)\,dx\right).\\
\intertext{If we now apply H\"older's inequality with exponents
  $\sigma$ and $2\sigma$, and then $2$, we get}
  & = \frac{C_0^2\gamma}{4}
    \left(
    \|f\|_{L^{\sigma'}(v;\Omega)}\|\psi^2\|_{L^{\sigma}(v;\Omega)}
    + 2
    \|f\|_{L^{\sigma'}(v;\Omega)}\|\psi\|_{L^{\sigma}(v;\Omega)}\right) \\
  & \leq \frac{C_0^2\gamma}{4}
    \left( \|\psi\|_{L^{2\sigma}(v;\Omega)}^2
    + 2
    \|\psi\|_{L^{2\sigma}(v;\Omega)}v(\Omega)^{\frac{1}{2\sigma}}\right). 
\end{align*}

If we now fix $\gamma\in (0,\frac{4}{C_0^2})$, then we can
re-arrange terms to get
\begin{equation} \label{eqn:psi-est}
\|\psi\|_{L^{2\sigma}(v;\Omega)}
\leq \frac{C_0^2\gamma}{2(1-\frac{C_0^2\gamma}{4})}
v(\Omega)^{\frac{1}{2\sigma}}.
\end{equation}
Therefore, again by H\"older's inequality and by~\eqref{eqn:psi-est}
applied twice, we have that
\begin{align*}
  \int_\Omega e^{\gamma u(x)}v(x)\,dx
  & = \int_\Omega \big(\psi(x)^2 +2\psi(x)\big)v(x)\,dx
    + v(\Omega) \\
  & \leq \|\psi\|_{L^{2\sigma}(v;\Omega)}^2v(\Omega)^{\frac{1}{\sigma'}}
  + 2 \|\psi\|_{L^{2\sigma}(v;\Omega)}v(\Omega)^{\frac{1}{(2\sigma)'}}
  + v(\Omega) \\
  & \leq C(\gamma,C_0)v(\Omega) \\
  & = M(\gamma,C_0,v(\Omega)).
\end{align*}
\end{proof}
%

%%%%%%%%%%%%%%%%%%%%%%
%%PROOF OF MAIN 1%%%%%
%%%%%%%%%%%%%%%%%%%%%%

\section{Proof of Theorem \ref{main0}}
\label{section:main0}

Fix $f \in L^A(v;\Omega)$ and let
${\bf u}=(u,\nabla u)\in QH^1_0(\Omega)$ be a non-negative weak
subsolution of \eqref{dp}.  We may assume without loss of generality
that $\|f\|_{L^A(v;\Omega)} >0$ (equivalently, that $f$ is non-zero on
a set $E\subset\Omega$ with $v(E)>0$); otherwise, a standard argument
shows that $u=0$ $v$-almost everywhere. (Cf.~\eqref{l2est} below.)  
By Lemma~\ref{SCALE},
$f\in L^{\sigma'}(v;\Omega)$.

For each $r>0$ define
$\varphi_r = (u-r)_+$ and let $S(r)=\{x\in\Omega~:~u(x) >r\}$.  Then by Lemma~\ref{lemma:main1-lemma}, 
 $(\varphi_r,\nabla \varphi_r) = ((u-r)_+,\mathbbm{1}_{S(r)}\nabla u )\in QH^1_0(v;\Omega)$.
We now estimate as follows: by the Sobolev inequality~\eqref{sob2}, the definition of a weak subsolution with $\varphi_r$ as
the test function, and H\"older's
inequality, we have that 
\begin{multline*}
  \|\varphi_r\|_{L^{2\sigma}(v;\Omega)}^2
  \leq C_0^2\int_{S(r)} |\sqrt{Q}\nabla \varphi_r|^2\,dx
= C_0^2\int_{S(r)} \nabla \varphi_r\cdot Q\nabla \varphi_r\,dx\\
= C_0^2\int_{S(r)} \nabla \varphi_r\cdot Q\nabla u\,dx
\leq C_0^2\int_{S(r)} f\varphi_r ~vdx
\leq C_0^2\|f\|_{L^{(2\sigma)'}(v;S(r))}\|\varphi_r\|_{L^{2\sigma}(v;\Omega)}
\end{multline*}
since $\nabla u =\nabla \varphi_r$ on $S(r)$.  If we divide through by
$\|\varphi_r\|_{L^{2\sigma}(v;\Omega)}$, we get
\begin{equation} \label{eqn:varphi-est}
\|\varphi_r\|_{L^{2\sigma}(v;\Omega)} \leq
C\|f\|_{L^{(2\sigma)'}(v;S(r))}.
\end{equation}

In order to estimate the norm of the right-hand side, recall that
since $\sigma>1$, $(2\sigma)' < \sigma'$, we can define the Young
function
$$B(t) = t^\frac{\sigma'}{(2\sigma)'}\log(e+t)^q.$$ 
It is immediate that $B_\sigma(t) =B(t^{(2\sigma)'})\preceq A(t)$ and
so by Lemma~\ref{holders}, a change of variables in the
Luxemburg norm, and Lemmas~\ref{normcompare} and~\ref{indicators} we get
\begin{align*}
  \|f\|_{L^{(2\sigma)'}(v;S(r))}^{(2\sigma)'}
  &=  \int_\Omega |f|^{(2\sigma)'}~\mathbbm{1}_{S(r)}v\,dx \\
  &\leq 2\|f ^{(2\sigma)'}\|_{L^B(v;\Omega)}
    \|\mathbbm{1}_{S(r)}\|_{L^{\bar{B}}(v;\Omega)}\\
  & = 2\|f \|_{L^{B_\sigma}(v;\Omega)}^{(2\sigma)'}
    \|\mathbbm{1}_{S(r)}\|_{L^{\bar{B}}(v;\Omega)}\\
  &\leq C\|f\|_{L^A(v;\Omega)}^{(2\sigma)'}
    \frac{v(S(r))^{\frac{1}{2\sigma-1}}}{\log(e+(v(S(r)))^{-1})^{q\left(\frac{(2\sigma)'}{\sigma'}\right)}},
\end{align*}
where  $C=C(\sigma,q,v(\Omega))$ is independent of  $f,\varphi,$ and ${\bf u}$.

\medskip

We now turn to our iteration argument.  For all $s>r$, $S(s)\subset S(r)$
and, for $x\in S(s)$, $ \varphi_r(x)>s-r>0$.  Hence,
if we combine the above two inequalities, we get
\begin{equation}\label{iterator}
  v(S(s))^\frac{1}{2\sigma}(s-r)
  \leq \|\varphi_r \mathbbm{1}_{S(s)}\|{L^{2\sigma}(v;\Omega)}
  \leq C\|f\|_{L^A(v;\Omega)} ~\displaystyle
  \frac{v(S(r))^\frac{1}{2\sigma}}
  {\log(e+(v(S(r)))^{-1})^\frac{q}{\sigma'}}.
\end{equation}

Define $r_0 = \tau_0\|f\|_{L^A(v;\Omega)}$ with $\tau_0$ to be chosen
below.  Our goal is to find $\tau_0$ sufficiently large so that
$v(S(r_0))=0$, as this immediately implies that
\[ \|u\|_{L^\infty(v;\Omega)} \leq \tau_0\|f\|_{L^A(v;\Omega)}, \]
which is what we want to prove.
To do this, we will use an iteration argument based on De Giorgi
iteration.    For each $k\in\mathbb{N}$ set
\begin{equation}\label{ck}
  C_k = r_0 (1-(k+1)^{-\epsilon})
\end{equation}
where $\epsilon>0$ will be chosen below, and let $C_0=C_1/2$.  The sequence
$\{C_k\}_{k=0}^\infty$  increases to $r_0$ and by an estimate using
the mean-value theorem we have that for each $k\in\mathbb{N}$,
\begin{equation}\label{Ckdist}
  C_{k+1}-C_k \geq \displaystyle\frac{\epsilon ~r_0}{(k+2)^{1+\epsilon}}.
\end{equation}
If we set  $s=C_{k+1},~r=C_k$, $\mu_k = v(S(C_k))$ in
inequality~\eqref{iterator}, we get 
\begin{equation} \label{eqn:mu-est}
  \mu_{k+1}
  \leq \left[\displaystyle\frac{C(k+2)^{1+\epsilon}}
    {\epsilon\tau_0}\right]^{2\sigma}
  \displaystyle\frac{\mu_k}
  {\log(e+\mu_k^{-1})^{\frac{2q\sigma}{\sigma'}}}
\end{equation}
for each $k\in\mathbb{N}$.  By the dominated convergence theorem
$\mu_k$ converges to $v(S(r_0))$, so to complete the proof we need to
prove that $\mu_k\rightarrow 0$.

Let $m_k = \log(\mu_k^{-1})$.  We will show that
$m_k\rightarrow \infty$ as $k\rightarrow \infty$, which is equivalent
to the desired limit.  To do so, we will show that we can choose
$\epsilon$ and $\tau_0$ such that $m_0\geq 2$ and 
\begin{equation} \label{eqn:best-est}
m_{k} \geq m_0 + k
\end{equation}
for all $k\in\mathbb{N}\cup\{0\}$.

Fix $\epsilon=\frac{q}{\sigma'}-1>0$.  
Since $2\sigma(1+\epsilon) = \frac{2\sigma q}{\sigma'}$, if we take
logarithms and re-arrange terms, 
inequality~\eqref{eqn:mu-est} becomes, for $k\in {\mathbb N}$,
\begin{equation}\label{goodest}
  m_{k+1} \geq
  2\sigma\log\left(\displaystyle\frac{\epsilon\tau_0}{C}\right)
  + \frac{2\sigma
    q}{\sigma'}\log\left(\displaystyle\frac{m_k}{k+2}\right)
  +m_k.
\end{equation}
The first step is to fix $m_0$ by an appropriate choice of
$\tau_0>0$.  
If we argue as we did to prove~\eqref{eqn:varphi-est} using $u$ as the test function in the
definition of a weak subsolution,  we get
\begin{equation}\label{l2est}
  \|u\|_{L^{2\sigma}(v;\Omega)} \leq C\|f\|_{L^{(2\sigma)'}(v;\Omega)}.
\end{equation}
If we estimate the right-hand side using H\"older's inequality and 
Lemma~\ref{SCALE},  we get
\[ \|f\|_{L^{(2\sigma)'}(v;\Omega)}
  \leq Cv(\Omega)^{\frac{1}{2\sigma}} \|f\|_{L^A(v;\Omega)}, \]
  where the constant $C$ is independent of $f$ and ${\bf u}$.  For each
  $x\in S(C_0)$ we have that $2u(x)/C_1>1$, so by H\"older's
  inequality and the above two estimates,
\begin{multline*}
  v(S(C_0))
  \leq
  \frac{2}{C_1}\int_{S(C_0)} uv\,dx
  \leq
  \frac{2}{C_1} \|u\|_{L^{2\sigma}(v;\Omega)}
  v(S(C_0))^{\frac{1}{(2\sigma)'} } \\
\leq 
\frac{2C}{C_1}  v(\Omega)^{\frac{1}{2\sigma}}\|f\|_{L^A(v;\Omega)}
v(S(C_0))^{\frac{1}{(2\sigma)'}}
  \leq
  \frac{2C v(\Omega)^{\frac{1}{2\sigma}}}{\tau_0(1-2^{-\epsilon})}
    v(S(C_0))^{\frac{1}{(2\sigma)'}}.
  \end{multline*}
If we re-arrange terms, we get
\begin{equation*}
  v(S(C_0)) \leq
  \left(\frac{C}{\tau_0(1-2^{-\epsilon})}\right)^{2\sigma},
\end{equation*}
where again the constant $C$ is independent of $f$ and $\bf u$.
Now choose $\tau_0>0$ so that 
\begin{equation}\label{tau_0res}
  \mu_0=v(S(C_0)) < e^{-2}, \; \text{ and } \;
\tau_0\geq \max\bigg\{ \frac{2^{\epsilon+1}eC}{2^\epsilon-1},
\frac{eC}{\epsilon}\bigg\}, 
\end{equation}
where $C$ is as in \eqref{iterator}.  Note that $\tau_0$ is
independent of ${\bf u}$ and $f$, and the first inequality
implies that $m_0\geq 2$.  

It is clear that $m_0 \geq m_0$ but for the sake of clarity we also
show that $m_1>m_0+1$.  Since $k=0$ we cannot use \eqref{goodest}, but
instead use \eqref{iterator} directly.  If we set $s=C_1$ and $r=C_0$
we find
$$\frac{C_1}{2}\mu_1^\frac{1}{2\sigma}
\leq C\|f\|_{L^A(v;\Omega)}
\frac{\mu_0^{\frac{1}{\ell(2\sigma)'}}}{\log(e+\mu_0^{-1})^\frac{q}{\sigma'}}.$$
If we use the definition of $C_1$ and recall that
$m_j=\log(\mu_j^{-1})$, we get 
\begin{multline*}
  m_1
  \geq
  2\sigma\log\left(
    \frac{(2^{\epsilon}-1)\tau_0}{2^{\epsilon+1}C}\right)
  +m_0 +2q(\sigma-1)\log(m_0)\\
  \geq \log\left( \frac{(2^{\epsilon}-1)\tau_0}{2^{\epsilon+1}eC}\right)
  +m_0 +1 \geq m_0+1;
\end{multline*}
the second inequality follows since $m_0\geq 1$,  and the third by our
choice of $\tau_0$.

Now suppose that $m_j \geq m_0 + j $ for some
$j\in\mathbb{N}$.  Since $m_0\geq 2$, \eqref{goodest} and
\eqref{tau_0res} together show that 
\begin{multline*}
  m_{j+1}
  \geq 2\sigma\log\left(\frac{\epsilon\tau_0}{C}\right)
  + \frac{2\sigma q}{\sigma'}\log\left(\frac{2+j}{2+j}\right) + m_0 + j \\
\geq \log\left(\frac{\epsilon\tau_0}{eC}\right) + m_0 + j + 1
\geq m_0+j+1.
\end{multline*}
Hence, by induction we have that inequality~\eqref{eqn:best-est} holds for
all $k$, and this completes our proof.

%%%%%%%%%%%%%%%%%%%%%%%%%%%%%%%
%%PROOF OF MAIN 2%%%%%%%%%%%%%%
%%%%%%%%%%%%%%%%%%%%%%%%%%%%%%%

\section{Proof of Theorem \ref{main1}}
\label{section:main1}

Our proof requires one technical lemma.

\begin{lem}\label{Gamma}
  Given $\sigma>1$, there exist constants $b\in (\sigma,2\sigma)$,
  $\bar{b}\in((2\sigma)',\sigma')$, and $p>1$ such that
\begin{equation} \label{eqn:gamma1}
\frac{1}{b}+\frac{1}{\bar{b}}+\frac{1}{p} = 1,
\end{equation}
and
\begin{equation} \label{eqn:gamma2}
\Gamma =
\frac{2\sigma}{\bar{b}}\left(\frac{\sigma'-\bar{b}}{\sigma'}
  +\frac{2\sigma - b}{2\sigma}\right)= 1.
\end{equation}
\end{lem}

\begin{proof}
We will first show that we can choose $b$ and $\bar{b}$ so that
\eqref{eqn:gamma2} holds, and then show that we can refine our choice
so that~\eqref{eqn:gamma1} holds as well.

Set $b = 2\sigma(1-\beta)$ and $\bar{b} = (1+\beta)(2\sigma)'$,
where 
$0<\beta < \min( \frac{1}{2}, \frac{\sigma'-(2\sigma)'}{(2\sigma)'})$
will be determined below.  With this restriction on $\beta$ it is
immediate that $b$ and $\bar{b}$ lie in the specified intervals.  
Moreover, if we insert these values into the definition of $\Gamma$,
we get
\begin{multline*}
  \Gamma
  = \frac{2\sigma}{(1+\beta)(2\sigma)'}
  \bigg( \frac{\sigma'-(1+\beta)(2\sigma)'}{\sigma'}
  + \frac{2\sigma -2\sigma(1-\beta)}{2\sigma}\bigg) \\
  =\frac{2\sigma}{(1+\beta)(2\sigma)'}
  \bigg((1+\beta)\bigg(1- \frac{(2\sigma)'}{\sigma'}\bigg)\bigg)
  = 2\sigma\bigg(\frac{1}{(2\sigma)'}-\frac{1}{\sigma'}\bigg) = 1.
\end{multline*}
This gives~\eqref{eqn:gamma2}.

To show that we can choose $p>1$ and $\beta$ so
that~\eqref{eqn:gamma1} holds, note that
\[
  \frac{1}{b} + \frac{1}{\bar{b}}
= \frac{1}{2\sigma(1-\beta)} + \frac{2\sigma-1}{2\sigma(1+\beta)}
= \frac{1+\beta + 2\sigma - 1 -2\beta\sigma+\beta}{2\sigma(1-\beta^2)}
= \frac{\sigma-\beta\sigma + \beta}{\sigma(1-\beta^2)}.
\]
Thus, $\frac{1}{b}+\frac{1}{\bar{b}} < 1$ exactly when $0<\beta <
\frac{1}{\sigma'}$.  Hence, if we choose $\beta$ sufficiently small we
can find $p>1$ such that ~\eqref{eqn:gamma1} holds.
\end{proof}

\begin{rem}
  In the proof of Lemma~\ref{Gamma}, the range of possible values for
  $\beta$ shrinks as the dimension increases.  In the classical case,
  $\sigma'=\frac{n}{2}$, and this value is generally a lower bound
  on $\sigma'$ in the more degenerate settings.  
\end{rem}

\medskip

\begin{proof}[Proof of Theorem~\ref{main1}]
  Let ${\bf u}=(u,\nabla u)\in QH_0^1(v;\Omega)$ be a non-negative weak subsolution of
\eqref{dp}.  By the homogeneity of equation~\eqref{dp} and
inequality~\eqref{eqn:fixed}, to prove this result it will suffice to
assume that  $\|f\|_{L^{\sigma'}(v;\Omega)}= 1$ and prove that
\begin{equation} \label{eqn:hom}
  \|u\|_{L^\infty(v;\Omega)} \leq C[ 1+ \log(1+\|f\|_{L^A(v;\Omega)})]. 
\end{equation}

To prove \eqref{eqn:hom} we will apply an iteration argument very similar to
that in the proof of Theorem~\ref{main0}, but to the solution of an
auxiliary equation we which now define.
Given that $\|f\|_{L^{\sigma'}(v;\Omega)}= 1$, and since by Theorem~\ref{main0} $u$ is bounded in
$\Omega$, we can apply  Lemma \ref{expint} and fix 
$\gamma\in(0,\frac{4}{C_0^2})$ such that
\begin{equation}\label{wexpint}
  \int_\Omega e^{\gamma u(x)}v(x)\,dx \leq M(\gamma,C_0,v(\Omega)) = M.
\end{equation}

Define $h= e^{\gamma u /p}$ (where $p>1$ will be determined below) and
let $w = h-1$.  By Lemma~\ref{AuxProb},
$(w,\frac{\gamma}{p}h\nabla u)\in QH^1_0(v;\Omega)$ is a non-negative
weak subsolution of
\begin{eqnarray}\label{dpB}
\left\{ \begin{array}{rcll}
-\Div\left(Q\nabla w\right)&=&\alpha fhv&\textrm{for }x\in\Omega,\\
w&=&0&\textrm{for }x\in\partial\Omega.
\end{array}
\right.
\end{eqnarray}
For each $r>0$, let $\varphi_r = (w-r)_+$ and
$S(r) = \{ x\in \Omega~:~ w(x)>r\}$.  By
Lemma~\ref{lemma:main1-lemma}, 
$(\varphi_r,\nabla \varphi_r)\in QH^{1}_0(v;\Omega)$.  By
Lemma~\ref{Gamma}, there exist
$\bar{b}\in ((2\sigma)',\sigma'),\,b\in (\sigma,2\sigma)$, and $p>1$ such
that~\eqref{eqn:gamma1} holds.  
We can now argue as we did in the proof of Theorem~\ref{main0} with
$\varphi_r$ as a test function, and then apply H\"older's inequality
twice to get
\begin{align}\label{2-10}
  \|\varphi_r\|_{L^{2\sigma}(v;\Omega)}^2
  &\leq  C\int_{S(r)} \nabla \varphi_r Q\nabla\varphi_r~dx  \nonumber \\
& = C\int_{S(r)} \nabla \varphi_r Q \nabla w~dx  \nonumber \\
& \leq C\int_{S(r)} f\varphi_r h ~vdx \nonumber \\
& \leq C\|f {\mathbbm 1}_{S(r)}\|_{L^{\bar{b}}(v;\Omega)}
\|\varphi_r\|_{L^{{b}}(v;\Omega)}\|h\|_{L^{p}(v;\Omega)} \nonumber \\
& \leq C\|f {\mathbbm 1}_{S(r)}\|_{L^{\bar{b}}(v;\Omega)}
\|\varphi_r\|_{L^{2\sigma}(v;\Omega)} v(S(r))^{\frac{2\sigma-b}{2\sigma}};
\end{align}
the last inequality follows since $b<2\sigma$ and since
by~\eqref{wexpint}, $h\in L^p(v;\Omega)$ with a constant independent
of $\bf u$ and $f$.

Now define the Young function $B(t) = t^\frac{\sigma'}{\bar{b}}\log(e+t)^q$
and note that $B(|t|^{\bar{b}})\preceq A(t)$.  Therefore, arguing as
before, by
Lemma~\ref{indicators} and \eqref{2-10} we have that
\[ 
  \|\varphi_r\|_{2\sigma}
  \leq
  C\|f\|_A
  \frac{v(S(r))^{\frac{1}{\bar{b}(\sigma'/\bar{b})'}+\frac{2\sigma -
        b}{2\sigma\bar{b}}}}
  {\log(e+v(S(r)^{-1}))^\frac{q}{\sigma'}}
= C\|f\|_A
\frac{v(S(r))^{\frac{\sigma'-\bar{b}}{\bar{b}\sigma'}+\frac{2\sigma - b}
    {2\sigma\bar{b}}}}{\log(e+v(S(r))^{-1})^\frac{q}{\sigma'}}
\]
We can now argue as we did in the proof of Theorem~\ref{main0} to get
that  for all $s>r$,
\[ 
  v(S(s))
  \leq
  \left(\frac{C\|f\|_A}{(s-r)}\right)^{2\sigma}
  \frac{v(S(r))^{\frac{2\sigma}{\bar{b}}\left(\frac{\sigma'-\bar{b}}{\sigma'}+\frac{2\sigma
          -
          b}{2\sigma}\right)}}{\log(e+v(S(r))^{-1})^\frac{2q\sigma}{\sigma'}}
  =  \left(\frac{C\|f\|_A}{(s-r)}\right)^{2\sigma}
  \frac{v(S(r))}{\log(e+v(S(r))^{-1})^\frac{2q\sigma}{\sigma'}};
\] 
the last inequality holds by~\eqref{eqn:gamma2}.  

We continue the proof of Theorem \ref{main0} and define
$\epsilon = \frac{q}{\sigma'}-1>0$, $C_k$, $k\geq 0$, as in
\eqref{ck}, and $m_k = -\log(v(S(C(k)))$ to again get the iteration inequality
\begin{equation} \label{eqn:iterate-again}
  m_{k+1} \geq 2\sigma \log\left(\frac{\epsilon \tau_0}{C}\right)
  + \frac{2\sigma q}{\sigma'}\log\left(\frac{m_k}{k+2} \right) + m_k.
\end{equation}
We will again prove that we can choose the parameter $\tau_0$ such
that $m_0>1$ and  for every $k\in\mathbb{N} \cup \{0\}$,
\begin{equation} \label{eqn:final-iteration}
  m_k \geq m_0 + k
\end{equation}

Assume for the moment that~\eqref{eqn:final-iteration} holds.  Then
arguing as before we have that $\|w\|_\infty \leq \tau_0\|f\|_A$:  that is,
$$e^{c||u||_\infty} \leq \tau_0(\|f\|_A +1),$$
which in turn implies that~\eqref{eqn:hom} holds as desired.

\medskip

Therefore, to complete the proof we need to show
that~\eqref{eqn:final-iteration} holds.  The proof is almost identical
to the proof of~\eqref{eqn:best-est}:  the only difference is in the
choice of $m_0$ which we will describe.  We first estimate as we did for
inequality~\eqref{2-10}:
\begin{multline*}
\|w\|^2_{L^b(v;\Omega)}
  \leq
  \|w\|^2_{L^{2\sigma}(v;\Omega)}v(\Omega)^{\frac{2\sigma-b}{\sigma}}
  \leq C\int_\Omega fwh\,vdx\;
  v(\Omega)^{\frac{2\sigma-b}{\sigma}} \\
  \leq C\|f\|_{L^{\bar{b}}(v;\Omega) }\|w\|_{L^b(v;\Omega)}
    \|h\|_{L^p(v;\Omega)}  v(\Omega)^{\frac{2\sigma-b}{\sigma}}
    \leq C\|f\|_{L^{\bar{b}}(v;\Omega)} \|w\|_{L^b(v;\Omega)}
      v(\Omega)^{\frac{2\sigma-b}{\sigma}}, 
\end{multline*}
    where the last inequality holds since $h\in L^p(v;\Omega)$ with
    norm bounded by a constant.  Furthermore, by H\"older's inequality
    and Lemma~\ref{SCALE},
    \[ \|f\|_{L^{\bar{b}}(v;\Omega)}^{\bar{b}}
      \leq \|f\|_{L^{\sigma'}(v;\Omega)}^{\bar{b}}
      v(\Omega)^{\frac{1}{(\sigma'/\bar{b})'}}
      \leq \|f\|_{L^{A}(v;\Omega)}^{\bar{b}}
      v(\Omega)^{\frac{1}{(\sigma'/\bar{b})'}}. \]
Since 
  $C_0 = C_1/2$,  for every $x\in S(C_0)$
  we have $\frac{2w(x)}{C_1} > 1$.  Thus, combining the above
  inequalities, we get
  \begin{multline*}
    v(S(C_0)) \leq
    \frac{2}{C_1}\int_{S(C_0)} w ~vdx
    \leq \frac{2}{C_1}\|w\|_{L^b(v;\Omega)}
    v(S(C_0))^{\frac{1}{\bar{b}}} \\
    \leq  \frac{2C}{C_1}  \|f\|_{L^{A}(v;\Omega)}
    v(S(C_0))^{\frac{1}{\bar{b}}} 
    v(\Omega)^{\frac{1}{\bar{b}(\sigma'/\bar{b})'}+\frac{2\sigma-b}{\sigma}}
    = \frac{C}{\tau_0(1-2^{-\epsilon})}    v(S(C_0))^{\frac{1}{\bar{b}}}.
  \end{multline*}
Hence,
  $$v(S_0) \leq \left(\frac{C}{\tau_0(1-2^{-\epsilon})}\right)^{b},$$
and so we can choose $\tau_0>0$ independent of both ${\bf u},f$ such that 
\[
\mu_0=v(S(C_0)) < e^{-2}, \quad 
\tau_0\geq \max\bigg\{ \frac{eC}{1-2^{-\epsilon}},
\frac{eC}{\epsilon}\bigg\}
\]
where $C$ is as in \eqref{eqn:iterate-again}. We may now proceed
  exactly as in the proof of ~\eqref{eqn:best-est} to get
  that~\eqref{eqn:final-iteration} holds.  This completes our proof.
\end{proof}
%

%%%%%%%%%%%%%%%%%%%%%%%%%%%%%%%%%%%%%%%%%
%%COUNTER EXAMPLE%%%%%%%%%%%%%%%%%%%%%%%%
%%%%%%%%%%%%%%%%%%%%%%%%%%%%%%%%%%%%%%%%%

\section{Theorem~\ref{main0} is almost sharp}
\label{section:counter-example}

In this section we construct Example~\ref{example:not-sharp-but-close}
that shows that Theorem \ref{main0} is almost sharp in the case of the
Laplacian.  Our example is intuitively straightforward.  Let
our domain $\Omega\subset \mathbb{R}^n$, $n\geq 3$, be the unit ball
$B=B(0,1)$, and define
\[ f(x) = |x|^{-2}\log(e+|x|^{-1})^{-1}.  \]
Let $A(t) = t^{\frac{n}{2}}\log(e+t)^q$.  We will show that $f\in L^A(B)$ if
and only if $q<\frac{n}{2}-1$.  Moreover, we claim that, at least
formally, if $u$ is the solution of $\Delta u = f$ on $B$, then
$u(0)=\infty$.  For if we use the well-known fact that the Green's
function for the unit ball is $c_n|x|^{2-n}$, then
\[ u(0) = c_n \int_B |x|^{-n} \log(e+|x|^{-1})^{-1}\,dx = \infty. \]

To make this argument rigorous we must justify our use of Green's
formula which requires that the function $f$
be continuous on $B$.  To overcome this, we give an approximation
argument and show that the inequality
\[ \|u\|_{L^\infty(B)} \leq C \|f\|_{L^A(B)} \]
cannot hold with a uniform constant.  For each $k\geq 1$, let $\chi_k$
be a continuous, non-negative, radial function such that $\chi_k(x)=0$
if $|x|\leq 2^{-k-1}$, and $\chi_k(x)=1$ if $2^{-k} \leq x <1$.
Define $f_k=u_k$.  Each $f_k$ is continuous, and if $u_k$ is the solution to the
Dirichlet problem
\[ \begin{cases}
    \Delta u_k = f_k  & x\in B, \\
    u_k = 0 & x\in \partial B,
  \end{cases}
\]
then at the origin it is given by 
\[ u_k(0) =  c_n\int_B |x|^{2-n} f_k(x)\,dx
  \geq c_n \int_{2^{-k} \leq |x|<1} |x|^{-n}
  \log(e+|x|^{-1})^{-1}\,dx. \]
It is immediate that $u_k(0) \rightarrow \infty$ as
$k\rightarrow\infty$.
Since by monotonicity of the norm, $\|f_k\|_{L^A(B)} \leq
\|f\|_{L^A(B)}$, we have that the inequality 
\[ u_k(0) \leq \|u_k\|_{L^\infty(B)} \leq C\|f_k\|_{L^A(B)}\leq C\|f\|_{L^A(B)} \]
cannot hold with a uniform constant if $f\in L^A(B)$.

Therefore, to complete the proof, it will suffice to show
$f\in L^A(B)$ if and only if $q<\frac{n}{2}-1$.  By the definition of
the Luxemburg norm, it will suffice to show that $f(A) \in L^1(B)$.
But this is straightforward:
\begin{align*}
  A(f(x))
  & = f(x)^{\frac{n}{2}} \log(e+f(x))^q \\
  & = x^{-n}\log(e+|x|^{-1})^{-\frac{n}{2}}
    \log(e+|x|^{-2}\log(e+|x|^{-1})^{-1})^{q} \\
  & \approx x^{-n}\log(e+|x|^{-1})^{-\frac{n}{2}} \log(e+|x|^{-1})^{q},
\end{align*}
where the implicit constant only depends  on $q$.
Thus,  $A(f)\in L^1(B)$ if and only if $\frac{n}{2}-q>1$, or
equivalently, $q<\frac{n}{2}-1$.

%%%%%%%%%%%%%%%%%%%%%%%%%%%%%%%%%%%%%%%%%%%%
%%%THE BIBLIOGRAPHY%%%%%%%%%%%%%%%%%%%%%%%%%
%%%%%%%%%%%%%%%%%%%%%%%%%%%%%%%%%%%%%%%%%%%%

\bibliographystyle{plain}
\bibliography{Infinity2Beyond}

\end{document}